\documentclass[a4paper,10pt]{amsart}

\usepackage{amsmath}
\usepackage{amssymb}
\usepackage[ansinew]{inputenc}

\def\Nset{\mathbb{N}}

\newtheorem{theorem}{Theorem}
\newtheorem{lemma}{Lemma}

\newtheorem{prop}{Proposition}

\date{}

\begin{document}

\title[M/G/1 Models with Admission Controls]{The Rate of Convergence to Stationarity for M/G/1 Models with Admission Controls via Coupling}
\author{Martin Kolb}
\address{University of Warwick, Department of Statistics,
Coventry, CV4 7AL, United Kingdom}
\email{M.Kolb@Warwick.ac.uk}

\author{Wolfgang Stadje}
\address{Institute  of Mathematics, University of Osnabr\"uck, Albrechtstr. 28A, 49076 Osnabr\"uck, Germany}
\email{wolfgang@mathematik.Uni-osnabrueck.de}

\author{Achim W\"ubker}
\address{Institute  of Mathematics, University of Osnabr\"uck, Albrechtstr. 28A, 49076 Osnabr\"uck, Germany}
\email{awuebker@Uni-Osnabrueck.de}

\maketitle

\begin{abstract}
We study the workload processes of two restricted M/G/1 queueing systems: in Model 1  
any service requirement that would
exceed a certain capacity threshold is truncated; in Model 2 
new arrivals do not enter the system if they have to wait more than a fixed threshold time in line. 
For Model 1 we obtain several 
results concerning the rate of convergence to equilibrium. 
In particular we derive uniform bounds for geometric ergodicity with respect to certain subclasses. 
However, we prove that for the class of all Model 1 workload processes there is no uniform bound. 
 For Model 2 we prove that geometric ergodicity follows from the finiteness of the moment-generating function of 
the service time distribution and derive bounds for the
convergence rates in special cases. The proofs use the coupling method.  
\end{abstract}
\section{Introduction}
In this paper we consider the long-run behavior of the workload processes $V_t$ of the two most important 
$M/G/1$ queueing systems with admission restrictions. 
We are interested in the rate of convergence toward the equilibrium (stationary) distribution $\pi$ 
and measure this rate in terms of the total variation distance, which is defined as
\begin{equation}\label{main_quantity}
d(x,t)=||\mathbb{P}_x \bigl(V_t \in\cdot\bigr)-\pi||_{TV}=\sup_{A\in\mathcal{B}}|\mathbb{P}_x 
\bigl(V_t\in A\bigr)-\pi(A)|,
\end{equation}
where of course $\mathbb{P}_{x}\bigl(V_t\in A\bigr)= \mathbb{P}\bigl(V_t\in A \mid 
V_0=x\bigr)$ and $\mathcal{B}$ is the Borel 
$\sigma$-field in $\mathbb{R}_+$.
The main purpose of this paper is the investigation of $d(x,t)$ as  
$t\rightarrow\infty$ for two $M/G/1$-type models. Let 
$T_{n}$ be the arrival time of the $n$th customer at the queue and $T_{0}=0$. 
The inter-arrival times $I_i=T_i-T_{i-1},i\in\mathbb{N}$, are assumed to be  i.i.d. and 
exponential with mean $1/\lambda$.  
Let $S_n$ be the service requirement of the $n$th customer; $(S_n)_{n\in\mathbb{N}}$ 
is assumed to be an i.i.d. sequence with common distribution $G$. 
\\
 
{\bf Model I: truncated service at the capacity limit}. 
 The workload process $V^{1,x}_t$ of this $M/G/1$ queue in a system with capacity 1
 is formally defined by 
\begin{equation}\label{M/G/1}
V^{1,x}_t=\left\{\begin{array}{cc}
x & t=0\\
\max[V^{1,x}_{T_{n-1}}-(t-T_{n-1}),0],& T_{n-1}\le t<T_n,n\ge 1\\
V^{1,x}_{T_n -}+ \min [S_n + V^{1,x}_{T_n -},1], &t=T_n,n\ge 1
\end{array}
\right.
\end{equation}
This model, which has been referred to as the 
"\textit{truncated service policy}" in the literature (see e.g. \cite{PeStaZa}), can be described 
as follows: whenever the 
total workload would increase beyond the capacity threshold, it is reduced such that this threshold is exactly 
reached but not exceeded. Note that under this rule every customer is admitted to the system.  
\\
 
{\bf  Model 2: bounded waiting time policy}. 
In the second model  
new arrivals whose waiting time
in line would exceed some constant are not admitted to enter the system. 
According to this policy, admission is interrupted 
as long as the workload process stays above the threshold, say 1. The workload process is thus given by   
\begin{equation}\label{defiV2}
V^{2,x}_t=\left\{\begin{array}{cc}
x & t=0\\
\max[V^{2,x}_{T_{n-1}}-(t-T_{n-1}),0],& T_{n-1}\le t<T_n,n\ge 1,\\
V^{2,x}_{T_n -}+S_{n}\mathbf{1}_{\{V^{2,x}_{T_n-}<1\}},&t=T_n,n\ge1. 
\end{array}
\right.
\end{equation}
Note that the distribution of $V^2_{x_t}$ has support $[0,\infty)$ if $G$ has unbounded support. 
\\

A comprehensive account of Model 1 for interarrival and service time distributions with rational 
Laplace-Stieltjes
transforms (LSTs) was already 
given by Cohen in his monograph \cite{Co} (Ch. III.5). 
His method is based on Pollaczek's classical contour integral equation which, in the
case of rational LSTs, leads to explicit, albeit very complicated formulas.
 In \cite{PeStaZa} the busy period distributions in the  $M/G/1$ and in the $G/M/1$ case 
are derived directly in terms of certain 
transforms of the underlying distributions. 
 Early papers on the waiting times in Model 2  
 are \cite{Da,C1,LT,GS,Ho}. In the more general context of
queues with state-dependent arrival and service rates some aspects of 
restricted $M/G/1$ queues were investigated in \cite{GK}. For other 
related models (e.g. partial refusal of overload work) see \cite{BPSZ}.  

Investigations concerning the rate of convergence to equilibrium for queueing systems have a long history, see 
e.g. \cite{KaMc2, Ca2, Ch, StPa, VaZe1, VaZePa, GaGo}. 
 Much of this work is based on the 
 spectral representation for birth and death processes due to Karlin and McGregor \cite{KaMc1}, 
whose application requires 
 exponentially distributed  service times,
so that this technique works well for $M/M/1$, $M/M/n$ and $M/M/\infty$, but 
is not applicable to $M/G/1$-type queues.

Our approach is based on the {\it coupling method}, 
which turns out to be flexible enough for dealing with general service distributions.   
In \cite{Th1, Th2, LuMeTw} coupling 
has been used to estimate convergence rates to equilibrium for standard $M/G/1$ queues without boundary modifications, 
but  
our construction is different.  
To the best of our knowledge,  convergence rates for the processes $V_t^{1,x}$
and $V_t^{2,x}$ defined above have not yet been derived.

The paper is organized as follows. 
In Section 2 we analyze the asymptotic behavior of $V_t^{1,x}$ for  
$t\rightarrow \infty$.  
We determine the density $\tilde{\pi}$ of the invariant distribution 
$\pi$ and give a new
formula for the distribution function of $\pi$. (Another expression was 
derived in \cite{Co} and \cite{Da} by different methods.)  
Then the general coupling method and the associated coupling inequality 
that will be used in this paper is presented.   
We show uniform ergodicity  with respect to the arrival rate  
 and to $G\in \mathcal{G}_{\rho,p}=\{G\in\mathcal{G}:G[\rho,\infty)\ge p\}$) 
($\rho, p >0 $ fixed) and also with respect to all service time distributions for fixed $\lambda>0$.  
However, uniformity fails to hold  over all $\lambda$ and $G$. 
At the end of Section 2 we discuss two examples. 
Section 3 is devoted to Model 2. We derive the invariant density, 
prove that geometric ergodicity follows from the finiteness of the moment-generating function of 
the service time distribution, and derive a bound for the
convergence rate in the case of bounded service times.
  
\section{Analysis of Model 1}\label{section_1}  
\subsection{The invariant distribution}
The Markov process 
$V_t^{1,x}$ is geometrically ergodic  
 and therefore has an uniquely 
determined invariant distribution $\pi$ 
satisfying  
\begin{equation}
d(x,t)=||\mathbb{P}\bigl(V_t^{1,x}\in \cdot\bigr)-\pi||_{TV}\le C_{x,\alpha} \exp(-\alpha t), \ \ \ t\ge 0,\ x\in[0,1] 
\end{equation}
for certain constants $\alpha >0$ and $C_{x,\alpha}>0$. 
To see this,
let $\tilde{T}_i$ be the time of the $i$th arrival of $V_t^{1,x}$ to $1$. Clearly 
$(V_{\tilde{T}_i+t}^{1,x})_{t\ge 0}$ has the same distribution
for all $i\in\mathbb{N}$, i.e., $1$ is a regenerative point. It follows from the general theory 
of regenerative processes (see e.g. \cite{As}, Ch. 6) that
if 
\begin{equation}\label{y1} Y_1=\tilde{T}_2-\tilde{T}_1
\end{equation}
is spread-out and $\mathbb{E}\bigl(Y_1\bigr)<\infty$, then the Markov process 
$V_t^{1,x}$ is geometrically ergodic with uniquely determined invariant distribution $\pi$. In our case 
 the spread-out condition as well as the finiteness of the 
expectation of $Y_1$ are clearly satisfied. Of course, $\pi$ is also the  asymptotic distribution 
of $V^{1,x}_t$ as $t\to \infty$ (see e.g. \cite{MeTw}, \cite{Nu}).

The invariant measure can be immediately written down in the form  
\begin{equation}\label{PiwithY}
\pi(A)=\frac{1}{\mathbb{E}_{1}\bigl(Y_1\bigr)}\int_{0}^{\infty}
\mathbb{P}\bigl(V_t^{1,1}\in A \mid Y_1>t\bigr)\mathbb{P}_{1}\bigl(Y_1>t\bigr)dt. 
\end{equation} 
Eq. \eqref{PiwithY} expresses $\pi$ in terms of the transient distributions 
of $V_t^{1,1}$; it is not very useful for explicit computations (except possibly for simulations).  
A formula expressing $\pi$ in terms of the system primitives $\lambda$ and $G$ is also well-known 
(see \cite{Co} and \cite{Da}): we have for the invariant distribution function 
\begin{equation}\label{Pi}
\pi(x)=\frac{\sum_{n=0}^{\infty}\int_{0}^{x}\dfrac{e^{\lambda(x-u)}[-\lambda(x-u)]^n}{n!}
dG_{n}(u)}{\sum_{n=0}^{\infty}\int_{0}^{1}\dfrac{e^{\lambda(x-u)}[-\lambda(x-u)]^n}{n!}dG_{n}(u)},\,\,\,0\le x\le 1,
\end{equation}
where $G_{n}$ is the $n$fold convolution of $G$ with itself and $\pi(x)$ is an abbreviation for $\pi[0,x]$.\\

A quick and neat direct approach leading to the {\it density} $\bar{\pi}$ of $\pi$ on $(0,1]$, 
and then via integration also to a new explicit formula for $\pi (x)$, is as follows. 
By the standard 
level crossing technique (see e.g. \cite{brill}),   $\bar{\pi} (x) $ is equal to the downcrossing rate of level $x$, 
which in turn is equal to the upcrossing rate of $x$. An upcrossing of $x$ occurs if for some $y\in  [0,x)$ 
a customer with a service requirement of size larger than $x-y$ arrives and the current workload is equal to $y$. Hence, 
setting $\bar{G}(x)=1-G(x)$,  
\begin{equation}\label{flow_out}
\bar{\pi} (x)=\int_{0}^{x}\mathbb{P}\bigl(S_1>x-y\bigr)\lambda\pi(dy)=\lambda\pi(0)\bar{G}(x)+\lambda 
(\bar{G}\ast\bar{\pi})(x).
\end{equation}
Iteration yields, for every $n\in \mathbb{N}$, 
\begin{eqnarray}\label{tildepi}
\bar{\pi}(x)&=&\lambda\pi(0)\bar{G}(x)+\lambda\bar{G}\ast(\lambda\pi(0)\bar{G}+\lambda\bar{G}\ast\bar{\pi})(x)\nonumber\\
&=&\ldots
=\pi(0)\sum_{i=1}^{n}\lambda^{i}\bar{G}^{\ast i}(x) + \lambda^n(\bar{G}^{\ast n}\ast\bar{\pi})(x). 
\end{eqnarray} 
Since the left-hand side of \eqref{tildepi} is finite and all terms are nonnegative 
it follows that $\sum_{i=1}^{\infty}\lambda^{i}\bar{G}^{\ast i}(x) <\infty$ and, consequently, 
$\lim_{n\to \infty} \lambda^n(\bar{G}^{\ast n}\ast\bar{\pi})(x)=0$. We thus obtain 
\begin{eqnarray}\label{tildepi1}
\bar{\pi}(x)&=& \pi(0)\sum_{i=1}^{\infty}\lambda^{i}\bar{G}^{\ast i}(x).  
\end{eqnarray} 
$\pi(0)$ can   be computed  by taking the integral on both sides:
\begin{equation}
1-\pi(0)= \pi(0)\sum_{i=1}^{\infty}\lambda^{i}\int_{0}^1\bar{G}^{\ast i}(x)dx.
\end{equation}
This yields
\begin{equation}\label{pi_0}
\pi(0)=\frac{1}{1+\sum_{i=1}^{\infty}\lambda^{i}\int_{0}^1\bar{G}^{\ast i}(x)dx}.
\end{equation} 
We have proved 
\begin{theorem}\label{invariant_density}
The density $\bar{\pi}$ of the invariant distribution $\pi$ for $x\in (0,1]$ is given by
\begin{equation}\label{invariant_density1}
\bar{\pi}(x)=\frac{1}{1+\sum_{i=1}^{\infty}\lambda^{i}\int_{0}^1\bar{G}^{\ast i}(x)dx}\sum_{i=1}^{\infty}\lambda^{i}\bar{G}^{\ast i}(x)
\end{equation}
and we have
\begin{equation}
\pi(x)=\frac{1+\sum_{i=1}^{\infty}\lambda^{i}\int_{0}^{x}\bar{G}^{\ast i}(y)dy}{1+\sum_{i=1}^{\infty}\lambda^{i}\int_{0}^1\bar{G}^{\ast i}(x)dx}.
\end{equation}
\end{theorem}

\subsection{The rate of convergence to equilibrium and the coupling inequality}\label{speed} 
We now prove that the process $V_t^{1,x}$ is uniformly 
geometrically ergodic, i.e., there exist constants $\alpha>0$ and $C=C_{\alpha}\in\mathbb{R}_{+}$  such that
\begin{equation}\label{uniform_ergodicity}
d(t):=\sup_{x\in[0,1]}d(x,t)=\sup_{x\in[0,1]}||\mathbb{P}\bigl(V_t^{1,x}
\in \cdot\bigr)-\pi||_{TV}\le C_{\alpha} \exp(-\alpha t).
\end{equation} 
In this case, 
\[
\rho:=\limsup_{t\rightarrow\infty}d(t)^{\frac{1}{t}}\le e^{-\alpha}<1
\]
and the quantity $1-\rho$ is called the {\it spectral gap}.
We say that $V_t$ satisfies the 
spectral gap property (SGP) if $\rho < 1$. 
Different characterizations of uniform and geometric ergodicity can be found in the monograph \cite{MeTw}.
For birth and death processes, sufficient conditions for geometric ergodicity are 
established in \cite{Va1} and a characterization
of uniform ergodicity in terms of the birth and death rates can be found in \cite{Ma} and the references therein.

 Let us describe the coupling method that will be used here. It can be easily established that     
 \begin{equation}\label{dt}
\sup_{x\in[0,1]}||\mathbb{P}\bigl(V_t^{1,x}\in 
\cdot\bigr)-\pi||_{TV}\le \sup_{x,y\in[0,1]}||\mathbb{P}\bigl(V_t^{1,x}
\in \cdot\bigr)-\mathbb{P}\bigl(V_t^{1,y}\in \cdot\bigr)||_{TV}=:\bar{d}(t).
\end{equation}
From the triangle inequality it follows that 
\begin{equation}\label{barbound}
\frac{\bar{d}(t)}{2}\le d(t)\le \bar{d}(t),
\end{equation} 
and thus the asymptotics of $d(t)$ can be obtained from by those of $\bar{d}(t)$. There are two main reasons for
considering $\bar{d}(t)$ instead of $d(t)$.   
First, it is known that $\bar{d}(t)$ is sub-multiplicative (see \cite{LePeWi}) in the sense that
\begin{equation}\label{subm}
\bar{d}(t+s)\le \bar{d}(t)\bar{d}(s).
\end{equation} 
Second, $\bar{d}(t)$ can be studied without any knowledge of $\pi$, although it determines the convergence rate of
$\mathbb{P}\bigl(V_t^{1,\cdot}\in \cdot\bigr)$ to $\pi$.
The value of $\bar{d}(t)$ can be bounded    
by using the following standard coupling inequality:  
We have  
\begin{eqnarray}
||\mathbb{P}\bigl(V_{t}^{1,x}\in\cdot\bigr)-\mathbb{P}\bigl(V_{t}^{1,y}\in\cdot\bigr)||_{TV}&=&\sup_{A\in\mathcal{B}}|\mathbb{P}\bigl(V_t^{1,x}\in A\bigr)-\mathbb{P}\bigl(V_t^{1,y}\in A\bigr)|\nonumber\\
&\le&\sup_{A\in\mathcal{B}}\mathbb{P}\bigl(\{V_t^{1,x}\in A\}\cap \{V_t^{1,y}\notin A\}\bigr)\nonumber\\
&\le&\mathbb{P}\bigl(V_t^{1,x}\not=V_t^{1,y}\bigr)\nonumber\\
&=&\mathbb{P}\bigl(T^{x,y}>t\bigr),
\end{eqnarray}
where $T^{x,y}=\inf\{t\ge 0:V_t^{1,x}=V_t^{1,y}\}$. This yields
\begin{equation}\label{coupling_inequalityI}
\bar{d}(t)\le\sup_{x,y}\mathbb{P}\bigl(T^{x,y}>t\bigr). 
\end{equation}
The strength of the above coupling inequality depends of course heavily on 
the choice of the coupling. In the following we
will consider processes $V_t^{1,x}$ and $V_t^{1,y}$ 
that are based on the same sequences $(T_i)_{i\in\mathbb{N}}$ and $(S_i)_{i\in\mathbb{N}}$ 
of arrival times and service requirements. We immediately see that 
that $V_t^{1,x}\le V_t^{1,y}$ whenever $x<y$; 
hence $V_t^{1,x}$ is a stochastically ordered Markov process in the sense of
\cite{LuMeTw}. This has the advantage that the coupling time $T^{x,y}$ can be related to
certain hitting times as has been done 
for example by Lund and Tweedie \cite{LuTw, LuMeTw, RoTw, ScTw}. 
These papers have been written in the setting of an 
unbounded state space, where uniform ergodicity mostly fails to be true, and  
focus on improving bounds that had been previously obtained by the Lyapunov function approach.  
Moreover, while in \cite{LuTw, LuMeTw, RoTw, ScTw} 
the tails of the coupling time are bounded from above by the tails of the hitting times of 
the "minimal element" of the state space, in our setting  
a simultaneous consideration of hitting the minimal or the  
maximal element leads to the desired bounds.  

Let us introduce the first times when the process that starts in $x$  hits $0$ or $1$, respectively: 
\begin{equation}
U_{0}^{x}:=\inf\{t\ge0:V_t^{1,x}=0\},\,\,\,\,U_{1}^{x}:=\inf\{t\ge0:V_t^{1,x}=1\}.
\nonumber 
\end{equation}
The following Lemma turns out to be very useful.

\begin{lemma}\label{main}
\begin{eqnarray}
\mathbb{P}\bigl(T^{x,y}>t\bigr)&\le&\mathbb{P}\bigl(U_{0}^{1}\wedge U_{1}^{0}>t\bigr)\label{coupling_ineq}\\
&\le& \min(\mathbb{P}\bigl(U_{0}^{1}>t\bigr),\mathbb{P}\bigl( U_{1}^{0}>t\bigr))\label{coupling_ineq_II}.
\end{eqnarray}
\end{lemma}
\begin{proof}
Since the coupling preserves the order, i.e., 
\begin{equation}
V_t^{1,x}\le V_t^{1,y}\,\,\forall x\le y,
\nonumber 
\end{equation}
it follows that $V_t^{1,y}=0$ implies $V_t^{1,x}=0$ and $V_t^{1,x}=1$ implies $V_t^{1,y}=1$.
\end{proof}
How much is lost when working with \eqref{coupling_ineq} and \eqref{coupling_ineq_II} as upper bounds for the tails of 
$\mathbb{P}\bigl(T^{x,y}>t\bigr)$? In Example 1 below 
 an application of \eqref{coupling_ineq} results in the 
exact rate of convergence to equilibrium, while \eqref{coupling_ineq_II} 
yields rates that are far from being optimal.  

We start with establishing uniform ergodicity for $V_t^{1,x}$.  
\begin{prop}\label{first_cor}
For all $t\ge 1$,  
\begin{equation}
\label{un} 
\sup_{x\in[0,1]}||\mathbb{P}(V_t^{1,x}\in\cdot)-\pi||_{TV}\le(1-e^{-\lambda})^{t}. 
\end{equation}
\end{prop}
\begin{proof}
Using the above coupling we obtain by applying Lemma \ref{main} that
\begin{equation}
\mathbb{P}\bigl(T^{x,y}>1\bigr)\le
 \mathbb{P}\bigl(T_{1}<1\bigr)=1-\mathbb{P}\bigl(T_{1}\ge1\bigr)=1-\mathbb{P}\bigl(U_{0}^1\le1\bigr)=1-e^{-\lambda}.
\end{equation}
Hence,  \eqref{coupling_inequalityI} yields 
\begin{equation}
\bar{d}(1)\le 1-e^{-\lambda}.
\end{equation}
Thus, by \eqref{subm}, 
\begin{equation}\label{simple}
\bar{d}(t)\le\bar{d}(1)^{t}=(1-e^{-\lambda})^{t}
\end{equation}
Now the assertion follows from \eqref{barbound} and \eqref{simple}.
\end{proof}
The bound in \eqref{un} becomes poor for large $\lambda$. 
However, in this case the process reaches level 1 
quickly so that  
 one might expect that \eqref{coupling_ineq_II} 
can be used to show,  for fixed $G$, uniform ergodicity with respect to $\lambda$. 
The following result shows that an even stronger statement holds.

\begin{theorem}\label{t:second}
For every $\beta> 0$ and $p>0$ the process $V_t^{1,x}$ has the SGP 
uniformly on $G\in\mathcal{G}_{\beta,p}=\{G\in\mathcal{G}:G(\beta,1]\ge p\}$ 
and uniformly in $\lambda$, i.e., for every $G\in\mathcal{G}_{\beta,p}$ 
and $\lambda >0$ the corresponding spectral gap $\rho =  \rho (G,\lambda)$ 
satisfies   
\begin{equation}\label{rho}
\rho (G,\lambda) \le 1-e^{-\lambda_0}<1,
\end{equation} 
where $\lambda_0=\lambda_0(p,\beta)$ is the unique solution $\lambda \in (1/p\beta ,\infty)$ of 
\begin{equation}\label{unique}
1-e^{-\lambda}=(\lambda \,p)^{\frac{1}{\beta}}e^{1-\lambda\,p}. 
\end{equation}
 
\end{theorem}
\begin{proof}
As in Proposition 1 we can easily derive that for all $\lambda >0$ we have
\begin{equation}\label{not_optimal}
\mathbb{P}\bigl(U^1_0>t)\le \mathbb{P}\bigl(U^1_0>1)^t\le(1-e^{-\lambda})^t
\end{equation}
However, the right-hand side of \eqref{not_optimal} tends to zero as $\lambda\rightarrow\infty$ and hence does not yield
any uniformity.
Consider an arbitrary $G\in \mathcal{G}_{\beta,p}$ and define the process
$\hat{V}_t^{1,x}$ as $V_t^{1,x}$ with the difference that 
\begin{itemize}
\item
All jumps of size $<\beta$ are not recognized
\item
All jumps of size $\ge\beta$ are reduced to size $\beta$.
\end{itemize}
Observe that the arrival times of the jumps of size $\beta$ 
form a Poisson process with intensity $\lambda p$ and that, 
obviously, $\hat{V}_t^{1,x}\le V_t^{1,x}$ for all $t\in\mathbb{R}_{+}$. Now let $\hat{U}_1^x$ be defined as $U_1^x$ but 
referring to
$\hat{V}_t^{1,x}$ instead of $V_t^{1,x}$ in its definition. Then we have $\hat{U}_1^x>U_1^x$ and therefore 
\begin{eqnarray}\label{probably_optimal}
\mathbb{P}\bigl(U_{1}^x>t\bigr)&\le&\mathbb{P}\bigl(U_{1}^0>t\bigr)\le\mathbb{P}\bigl(\hat{U}_1^0>t\bigr)\nonumber\\
&\le& \mathbb{P}\bigl(\mbox{ less than }\lceil\frac{1+t}{\beta}\rceil
\mbox{ jumps of size at least }\beta \mbox{ occur up to time }t\bigr)\nonumber\\
&=&1- \sum_{i=\lceil\frac{1+t}{\beta}\rceil}^{\infty}e^{-\lambda p t}\frac{(\lambda p t)^{i}}{i!}=\sum_{i=0}^{{\lceil\frac{1+t}{\beta}\rceil}-1}e^{-\lambda p t}\frac{(\lambda p t)^{i}}{i!}. 
\end{eqnarray}
Now Lemma \ref{main} yields
\begin{equation}\label{to_be_optimized}
\mathbb{P}\bigl(T^{x,y}>t\bigr)\le 
\min\bigl[(1-e^{-\lambda})^t,\sum_{i=0}^{{\lceil\frac{1+t}{\beta}\rceil}-1}e^{-\lambda p t}\frac{(\lambda p t)^{i}}{i!}\bigr]
\end{equation}
and hence for $\rho=\rho (G,\lambda)$
\begin{eqnarray}\label{to_be_optimized_II}
\rho (G,\lambda )
&=&\limsup_{t\rightarrow\infty}d(t)^{\frac{1}{t}}
\le\limsup_{t\rightarrow\infty}\mathbb{P}\bigl(T^{x,y}>t\bigr)^{\frac{1}{t}}\nonumber\\
&\le& \min\bigg( 1-e^{-\lambda},\limsup_{t\rightarrow\infty}
\Big(\sum_{i=0}^{{\lceil\frac{1+t}{\beta}\rceil}-1}e^{-\lambda p t}\frac{(\lambda p t)^{i}}{i!}\Big)^{\frac{1}{t}}\bigg)
\nonumber\\
&=& \min \Big(1-e^{-\lambda},\mathbf{1}_{\{\lambda p \le 1\}}+\mathbf{1}_{\{\lambda p > 1\}}(\lambda \,p)^{\frac{1}{\beta}}
e^{1-\lambda\,p}\Big). 
\end{eqnarray}
Let us consider the 
right-hand side of \eqref{to_be_optimized_II}: While $\lambda \mapsto 1-e^{-\lambda}$, $\lambda \in (0,\infty)$,  
is strictly increasing from 0 to 1, the function 
$\lambda \mapsto 
\mathbf{1}_{\{\lambda p \le 1\}}+\mathbf{1}_{\{\lambda p > 1\}}(\lambda \,p)^{\frac{1}{\beta}}e^{1-\lambda\,p}$ 
equals $1$ for $\lambda p\le 1$, is strictly increasing for $1< \lambda\,p\le\frac{1}{\beta}$
to a value larger than 1 
and strictly decreasing to 0 for $\lambda\,p>\frac{1}{\beta}$. This implies that there exists a unique 
$\lambda_0 \in (1/p\beta ,\infty)$ for which
\eqref{unique} holds true, and this $\lambda_0$  satisfies \eqref{rho}.
  


\end{proof}


{\bf Remarks}.
1. Observe that \eqref{to_be_optimized_II} yields a lower bound 
for the spectral gap $1-\rho$ for every given triple $\lambda,p,\beta$. \\
2.  
Since $\lambda_0>\frac{1}{p\,\lambda}$, 
the above lower bound for the spectral gap converges to 0 for fixed $p>0$ and $\beta\rightarrow 0$. 
 Below we deal with the 
question 
whether geometric ergodicity holds uniformly on the set of {\it all} service time distributions.\\  
3. As another approach to compute an upper bound, one could try the following:
\begin{eqnarray}\label{not_easy_to_handle}
&&\mathbb{P}\bigl(U_{1}^{0}<t\bigr)\le\mathbb{P}\bigl(\sup_{s\le t}V_t^{1,0}<1\bigr)\nonumber\\
&\le&\sum_{i=0}^{\infty}\mathbb{P}\bigl(\sup_{s\le t}V_t^{1,0}<1|J_{t}=i\bigr)\mathbb{P}\bigl(J_{t}=i\bigr)\nonumber\\
&=&\sum_{i=0}^{\infty} 
 e^{-\lambda t} \lambda^{i}\int_{[0,1]^{i}}\mathbf{1}_{\{x_1<x_2<\ldots<x_i\le t\}}\nonumber\\
&&\mathbb{E}\bigl[[\,\,[\ldots[[S_1-(x_2-x_1)]^{+}+S_2-(x_3-x_2)]^{+}+\ldots]
\nonumber\\ && \hspace{5cm} 
+S_i-(t-x_i)]^{+}\bigr]dx_1\ldots dx_i.
\end{eqnarray}
However, the calculation of the integral in \eqref{not_easy_to_handle} seems to be difficult. 

\subsection{Some special cases}\label{examples}
Let us consider two examples in which Lemma \ref{main} can be used directly. 
The first example exhibits a surprising behavior.

{\bf Example 1.} Assume that the service time distribution $G$ has its support in 
$[1,\infty)$. Consequently, whenever a customer enters the system 
both processes $V^{1,x}_t$ and  $V^{1,y}_t$ merge immediately  
and then remain together forever. On the other hand, if no customer 
enters the system during the first unit of time, both processes  
arrive at state $0$ independently of the initial values $x$ and $y$. Consequently,
\begin{eqnarray}
\mathbb{P}\bigl(T^{x,y}>t\bigr)=\mathbb{P}\bigl(T^{x,y}>t,T_1\le t\bigr)+\mathbb{P}\bigl(T^{x,y}>t,T_1> t\bigr)
\le e^{-\lambda t}\mathbf{1}_{[0,1)}(t).
\end{eqnarray} 
In particular we have $\bar{d}(t)=0$ for $t\ge 1$ and hence $d(t)=0$ for $t\ge 1$ by \eqref{barbound}. 
The fast speed of convergence is quite surprising, since it means that the process is 
already in equilibrium after one unit of time regardless of its initial value. 
This result shows the power of the simple coupling inequality \eqref{coupling_inequalityI}.\\

What is the distribution $\pi$ of $V^{1,x}_1\,\,$? Since 
\[
\bar{G}(x)=
\left\{
\begin{array}{c c}
1,&\,\,x\in[0,1)\\
0,&x \ge 1
\end{array}
\right.
\]
a straightforward calculation shows that  
\begin{equation}\label{sf1}
\sum_{i=1}^{\infty}\lambda^{i}\bar{G}^{\ast i}(x)=\lambda x-\lambda +\lambda e^{\lambda x}
\end{equation}
and hence
\begin{equation}\label{sf2}
\sum_{i=1}^{\infty}\lambda^{i}\int_0^1 \bar{G}^{\ast i}(x)dx=e^{\lambda}-\frac{\lambda}{2}-1.
\end{equation}
Now insert \eqref{sf1} and \eqref{sf2} in \eqref{invariant_density1}. This yields
\begin{equation}
\bar{\pi}(x)=\frac{\lambda x -\lambda +\lambda e^{\lambda x}}{e^{\lambda}-\frac{\lambda}{2}}.
\end{equation}
Adding the atom at 0 it is readily seen that the distribution function $\pi (x)$ is given by 
\[
\pi(x)=\frac{e^{\lambda x}-\lambda x+\frac{\lambda}{2}x^2}{e^{\lambda}-\frac{\lambda}{2}}.
\]   

{\bf Example 2}. 
Assume that $p=\mathbb{P}\bigl(G\ge 1\bigr)>0$. Then we have
\begin{equation}\label{ef}
\sup_{x\in[0,1]}||\mathbb{P}(V_t^{1,x}\in\cdot)-\pi||_{TV}\le e^{-\lambda p t}.
\end{equation}
To see this, we use use the same coupling as before.
Whenever a jump of size larger than one occurs, both processes glue together regardless 
of their initial values. The arrival times of 
the jumps of size larger than one is a Poisson 
process with intensity $\lambda p$. Hence \eqref{ef} follows from 
\[
\mathbb{P}\bigl(T^{x,y}>t\bigr)\le\mathbb{P}\bigl(T_1>t\bigr)\le e^{-p\lambda t}.
\] 
On the other hand, we have that 
\begin{eqnarray}
\mathbb{P}\bigl(T^{x,y}>1\bigr)&\le& \mathbb{P}\bigl(\text{ for $t\in[0,1]$ 
the process has at least one jump of size $<1$}\bigr)\nonumber\\
&=&1-e^{-\lambda (1-p)}
\end{eqnarray}
and hence
\begin{equation}\label{ef_2}
\sup_{x\in[0,1]}||\mathbb{P}(V_t^{1,x}\in\cdot)-\pi||_{TV}\le (1-e^{-\lambda (1-p)})^t.
\end{equation}
Now \eqref{ef} and \eqref{ef_2} together yield the following lower bound for spectral gap: 
\[
1-\rho\ge \min(1-e^{-\lambda p },e^{-\lambda (1-p)}).
\]
It follows immediately that $e^{-\lambda_0 (1-p)}$ is a lower bound which is uniform in $\lambda$,
where $\lambda_0$ is the unique solution of $e^{-\lambda (1-p)}=1-e^{-\lambda p}$.

\begin{subsection}{The SGP does not hold uniformly}\label{general} 
Let $\mathcal G$ be the set of all distributions on $(0,\infty)$. For general service distribution $G
\in \mathcal{G}$, it is not easy to analyze 
the time when the processes $V_t^{1,x}$ and $V_t^{1,y}$ merge. 

We show now that there is no universal bound for the spectral gap valid for
all $\lambda$ and all $G$. 
We will see in the proof of this result 
that the spectral gap converges to zero when taking  
the point mass at $\epsilon$ as service distribution, choosing 
$\lambda=\lambda_{\epsilon}\rightarrow\infty$ in a balanced way and letting $\epsilon\rightarrow 0$. 

\begin{theorem}\label{second_prop} 
\begin{equation}
\inf_{G\in  \mathcal{G},\lambda>0}(1-\rho (G,\lambda))=0.
\end{equation} 
\end{theorem}

\begin{proof}
Let $\epsilon>0$ and take $G=\delta_{\epsilon}$, the point mass at $\epsilon$. Then 
\begin{equation}
\mathbb{P}\bigl(T_1>\epsilon\bigr)=e^{-\lambda\epsilon}.
\nonumber 
\end{equation} 
Moreover, let
\begin{equation}
R_{0}^{(\epsilon,\lambda)}=0,\,R_1^{(\epsilon,\lambda)}=
\min\{\epsilon,T_1\},\,R_{i+1}^{(\epsilon,\lambda)}
=\min\{R_i^{(\epsilon,\lambda)}+\epsilon,\min \{T_j:T_j\ge R_i,j\in\mathbb{N}\}\},
\nonumber 
\end{equation}
where as before the $T_j$ denote the arrival times of the process. We have, for
$x\in(\epsilon,1-\epsilon)$, 
\begin{eqnarray}
\mathbb{P}_{x}\bigl(V_{R_1^{(\epsilon,\lambda)}}\in\cdot\bigr)&=&
\mathbb{P}_{x}\bigl(V_{R_1^{(\epsilon,\lambda)}}
\in\cdot \mid T_1\ge \epsilon\bigr)\mathbb{P}_{x}\bigl(
T_1\ge \epsilon\bigr)
\nonumber\\ 
&& \hspace{3cm} +\int_0^{\epsilon}\mathbb{P}_{x}\bigl(
V_{R_1^{(\epsilon,\lambda)}}\in\cdot \mid 
T_1= s\bigr)\mathbb{P}_{x}\bigl(T_1\in ds \bigr)\nonumber\\
&=&e^{-\lambda\epsilon}\delta_{x-\epsilon}+\int_0^{\epsilon}\delta_{x+\epsilon -s}(\cdot)\lambda e^{-\lambda s}ds.
\nonumber 
\end{eqnarray}
This implies that 
\begin{equation}
\mathbb{P}_{x}\bigl(V_{R_1^{(\epsilon,\lambda)}}
\in\cdot\bigr)\rightarrow \delta_{x-\epsilon},\,\,\,\lambda\rightarrow 0\mbox{ and }
\mathbb{P}_{x}\bigl(V_{R_1^{(\epsilon,\lambda)}}\in\cdot\bigr)\rightarrow \delta_{x+\epsilon},\,\,\,\lambda\rightarrow \infty,
\nonumber 
\end{equation}
where the convergence is with respect to the weak topology. In particular, 
\begin{equation}\label{expectation}
\mathbb{E}_{x}\bigl(V_{R_1^{(\epsilon,\lambda)}}\bigr)\rightarrow {x-\epsilon},\,\,\,\lambda\rightarrow 0 \mbox{ and }
\mathbb{E}_{x}\bigl(V_{R_1^{(\epsilon,\lambda)}}\bigr)\rightarrow x+\epsilon,\,\,\,\lambda\rightarrow \infty.
\end{equation}
Observe that $\mathbb{E}_{x}\bigl(V_{R_1^{(\epsilon,\lambda)}}\bigr)$ depends continuously on $\lambda$. Hence by 
the intermediate value theorem there exists a $\tilde{\lambda}$ such
that $\mathbb{E}_{x}\bigl(V_{R_1^{(\epsilon,\tilde{\lambda})}}\bigr)=x$.
We can write 
\begin{equation}\label{sum}
V_{R_n^{(\epsilon,\tilde{\lambda})}}=V_{R_1^{(\epsilon,\tilde{\lambda})}}
+\sum_{i=1}^{n-1}\bigl(V_{R_{i+1}^{(\epsilon,\tilde {\lambda})}}-V_{R_{i}^{(\epsilon,\tilde{\lambda})}}\bigr).
\end{equation}
Since the inter-arrival times are 
exponentially distributed, it follows that for fixed $n$ 
and sufficiently small $\epsilon\le\tilde{\epsilon}(n,x)$ the sum in 
\eqref{sum} is a sum of i.i.d. random variables with expectation zero. Here, $\tilde{\epsilon}(n,x)$ must be 
chosen such that the process started at $x$ cannot reach the boundary up to time $R_n$. 
Now let $\tilde{V}_{R_n^{(\epsilon,\tilde{\lambda})}}$ be the boundary-free 
version of $V_{R_n^{(\epsilon,\tilde{\lambda})}}$, i.e., 
let $\tilde{V}_{R_n^{(\epsilon,\tilde{\lambda})}}$ be 
defined analogously to $V_{R_n^{(\epsilon,\tilde{\lambda})}}$, where in the definition of 
$V^{1}$ we have to replace $\bar{S}_n$ by $S_n$. Moreover, let
\begin{equation}
M_n=M_n^{(\epsilon)}=\sum_{i=1}^{n-1}\bigl(\tilde{V}_{R_{i+1}^{(\epsilon,\tilde {\lambda})}}-\tilde{V}_{R_{i}^{(\epsilon,\tilde{\lambda})}}).
\end{equation}
Observe that $M_n$
is a martingale with respect to the filtration $\sigma(M_1,M_2,\ldots,M_n),n\in\Nset$.
Let $N_t:=\max\{i\in\mathbb{N}:R_i\le t\}$ and $h$ be a function such that $h(\epsilon)\rightarrow 0$ for 
$\epsilon\rightarrow 0$, but $h(\epsilon)/\epsilon^{\alpha}\rightarrow \infty$ 
for all $\alpha>0$. Then if $x$ satisfies $x\ge\frac{3}{4}+\epsilon$  and $0< \epsilon<\frac{1}{4}$ we obtain 
\begin{eqnarray}
\mathbb{P}_x\bigl(V_t^{(\epsilon,\tilde{\lambda})}<\frac{1}{2}\bigr)&=&\mathbb{P}_{x}\bigl(V_{R_1^{(\epsilon,\tilde{\lambda})}}+\sum_{i=1}^{N_t-1}\bigl(V_{R_{i+1}^{(\epsilon,\tilde{\lambda})}}-V_{R_i^{(\epsilon,\tilde{\lambda})}}\bigr)+V_t-V_{R_{N_t}}<\frac{1}{2}\bigr)\nonumber\\
&\le&\mathbb{P}_{x}\bigl(\sum_{i=1}^{N_t-1}\bigl(V_{R_{i+1}^{(\epsilon,\tilde{\lambda})}}-V_{R_i^{(\epsilon,\tilde{\lambda})}}\bigr)<\frac{1}{2}-x+\epsilon\bigr)\nonumber\\
&\le&\mathbb{P}_{x}\bigl(\sum_{i=1}^{N_t}\bigl(V_{R_{i+1}^{(\epsilon,\tilde{\lambda})}}-V_{R_i^{(\epsilon,\tilde{\lambda})}}\bigr)<-\frac{1}{4}\bigr)\nonumber\\
&\le&\mathbb{P}_{x}\bigl(\sum_{i=1}^{N_t}(V_{R_{i+1}^{(\epsilon,\tilde{\lambda})}}-V_{R_{i}^{(\epsilon,\tilde{\lambda})}})<-\frac{1}{4},N_t\le \frac{h(\epsilon)}{\epsilon^2}\bigr)+\mathbb{P}_{x}\bigl(N_t> \frac{h(\epsilon)}{\epsilon^2}\bigr)\nonumber\\
&\le&\mathbb{P}_{x}\bigl(\sum_{i=1}^{N_t\wedge \frac{h(\epsilon)}{\epsilon^2}}(V_{R_{i+1}^{(\epsilon,\tilde{\lambda})}}
-V_{R_{i}^{(\epsilon,\tilde{\lambda})}})<-\frac{1}{4}\bigr)+\mathbb{P}_{x}\bigl(N_t> \frac{h(\epsilon)}{\epsilon^2}\bigr)
\nonumber\\
&\le&\mathbb{P}_{x}\bigl(|\sum_{i=1}^{N_t\wedge \frac{h(\epsilon)}{\epsilon^2}}(V_{R_{i+1}^{(\epsilon,\tilde{\lambda})}}
-V_{R_{i}^{(\epsilon,\tilde{\lambda})}})|\ge\frac{1}{4}\bigr)+\mathbb{P}_{x}\bigl(N_t> \frac{h(\epsilon)}{\epsilon^2}\bigr). 
\nonumber
\end{eqnarray} Hence, 
\begin{eqnarray}
\mathbb{P}_x\bigl(V_t^{(\epsilon,\tilde{\lambda})} 
&\le&\mathbb{P}_{x}\bigl(\sup_{j\le\frac{h(\epsilon)}{\epsilon^2}}\|\sum_{i=1}^{j}(V_{R_{i+1}^{(\epsilon,\tilde{\lambda})}}-V_{R_{i}^{(\epsilon,\tilde{\lambda})}})\|\ge\frac{1}{4}\bigr)+\mathbb{P}_{x}\bigl(N_t> \frac{h(\epsilon)}{\epsilon^2}\bigr)\nonumber\\
&\le&\mathbb{P}_{x}\bigl(\sup_{j\le\frac{h(\epsilon)}{\epsilon^2}}\|\sum_{i=1}^{j}(V_{R_{i+1}^{(\epsilon,\tilde{\lambda})}}-V_{R_{i}^{(\epsilon,\tilde{\lambda})}})\|\ge 1-x\bigr)+\mathbb{P}_{x}\bigl(N_t> \frac{h(\epsilon)}{\epsilon^2}\bigr)\nonumber\\
&=&\mathbb{P}_{x}\bigl(\sup_{j\le\frac{h(\epsilon)}{\epsilon^2}}\|\sum_{i=1}^{j}(\tilde{V}_{R_{i+1}^{(\epsilon,\tilde{\lambda})}}-\tilde{V}_{R_{i}^{(\epsilon,\tilde{\lambda})}})\|\ge 1-x\bigr)+\mathbb{P}_{x}\bigl(N_t> \frac{h(\epsilon)}{\epsilon^2}\bigr)\nonumber\\
&\le& \frac{1}{(1-x)^2} \mathbb{E}\bigl(\|\sum_{i=1}^{\lfloor\frac{h(\epsilon)}{\epsilon^2}\rfloor}(\tilde{V}_{R_{i+1}^{(\epsilon,\tilde{\lambda})}}-\tilde{V}_{R_{i}^{(\epsilon,\tilde{\lambda})}})\|^2\bigr)+\mathbb{P}_{x}\bigl(N_t> \frac{h(\epsilon)}{\epsilon^2}\bigr)\label{doob}\\
&=& \frac{1}{(1-x)^2}\lfloor\frac{h(\epsilon)}{\epsilon^2}\rfloor \mathbb{E}\bigl((\tilde{V}_{R_{2}^{(\epsilon,\tilde{\lambda})}}-\tilde{V}_{R_{1}^{(\epsilon,\tilde{\lambda})}})^2\bigr)+\mathbb{P}_{x}\bigl(N_t> \frac{h(\epsilon)}{\epsilon^2}\bigr)\nonumber\\
&\le&\frac{1}{(1-x)^2}h(\epsilon)+\mathbb{P}_{x}\bigl(N_t> \frac{h(\epsilon)}{\epsilon^2}\bigr).\label{last}
\end{eqnarray}
In \eqref{doob} we have used Doob's maximal inequality for martingales. Next note that  
\begin{eqnarray}
\mathbb{P}_{x}\bigl(N_t> \frac{h(\epsilon)}{\epsilon^2}\bigr)&\le&\mathbb{P}\bigl(R_1+\sum_{i=1}^{\lfloor\frac{h(\epsilon)}{\epsilon^2}\rfloor-1}(R_{i+1}-R_{i})<t\bigr)\nonumber\\
&\le&\mathbb{P}\bigl(\sum_{i=1}^{\lfloor\frac{h(\epsilon)}{\epsilon^2}\rfloor-1}(R_{i+1}-R_{i})<t+\epsilon\bigr)\nonumber\\
&\le&\mathbb{P}\bigl(\frac{1}{\sigma_{\epsilon}\sqrt{\lfloor\frac{h(\epsilon)}{\epsilon^2}\rfloor}}
\sum_{i=1}^{\lfloor\frac{h(\epsilon)}{\epsilon^2}\rfloor-1}\bigl((R_{i+1}-R_{i})
-\mathbb{E}\bigl(R_{2}-R_{1}\bigr)\bigr)\label{minus_infinity}
\\ && \hspace{4cm} 
<\frac{t+\epsilon-\lfloor\frac{h(\epsilon)}{\epsilon^2}\rfloor\mathbb{E}
\bigl(R_{2}-R_{1}\bigr)}{\sigma_{\epsilon}\sqrt{\lfloor\frac{h(\epsilon)}{\epsilon^2}\rfloor}}\bigr), \nonumber  
\end{eqnarray}
where $\sigma_{\epsilon}^2=\mbox{Var}(R_2-R_1)$.
From the standard central limit theorem it follows that 
\[
\frac{1}{\sigma_{\epsilon}\sqrt{\frac{h(\epsilon)}{\epsilon^2}}}\sum_{i=1}^{\lfloor\frac{h(\epsilon)}{\epsilon^2}\rfloor-1}\bigl((R_{i+1}-R_{i})-\mathbb{E}\bigl(R_{2}-R_{1}\bigr)\bigr)\rightarrow N(0,1)
\]
in distribution. 
On the other hand, the right-hand side in \eqref{minus_infinity} converges to $-\infty$, and hence we have
\begin{equation}
\mathbb{P}_{x}\bigl(N_t> \frac{h(\epsilon)}{\epsilon^2}\bigr)\rightarrow 0\mbox{ for }\epsilon\rightarrow 0.
\end{equation}
This together with \eqref{last} implies that
\begin{equation}\label{right_concentration}
\mathbb{P}_{x}\bigl(V_t^{(\epsilon,\tilde{\lambda})}<\frac{1}{2}\bigr)\rightarrow 0\mbox{ for }\epsilon\rightarrow 0.
\end{equation}
Now we can carry out a similar calculation for $y\le \frac{1}{4}-\epsilon$ ($0<\epsilon<\frac{1}{4}$), yielding 
\begin{equation}\label{left_concentration}
\mathbb{P}_{y}\bigl(V_t^{(\epsilon,\tilde{\lambda})}>\frac{1}{2}\bigr)\rightarrow 0\mbox{ for }\epsilon\rightarrow 0.
\end{equation}
Let $\bar{d}_{\epsilon}(t)=\sup_{x,y}\|\mathbb{P}_{x}\bigl(V_t^{(\epsilon,\tilde{\lambda})}
\in\cdot\bigr)-\mathbb{P}_{y}\bigl(V_t^{(\epsilon,\tilde{\lambda})}\in\cdot\bigr)\|_{TV}$. Then 
it follows from \eqref{right_concentration} and \eqref{left_concentration} 
\[\bar{d_{\epsilon}}(t)\rightarrow 1,\,\,\,\epsilon \rightarrow 0\mbox{ for all }t>0,\]
from which the result follows.
\end{proof}
\end{subsection}

\begin{section}{Results for Model 2} 
In this section we present the basic analysis of Model 2.
It is shown that $\mathbb{E}(S_1)<\infty$ implies that the
process $V_t^{2,\cdot}$ has an invariant distribution $\pi$ 
and determine an explicit formula for $\pi$. A condition ensuring geometric ergodicity is given and an 
estimate for the rate of convergence in the case of bounded jumps is derived.

\begin{subsection}{The invariant distribution} 
\begin{theorem}
The process $V^{2,x}_t$ has an invariant distribution if $\mathbb{E}\bigl(S_1\bigr)<\infty$.
In this case the invariant density $\tilde{\pi}$ on $(0,\infty)$ is given by  
\begin{equation}
\tilde{\pi}(x)=\left\{
\begin{array}{c c}
\dfrac{\sum_{i=1}^{\infty}\lambda^{i}\bar{G}^{\ast i}(x)}{1+\lambda\mathbb{E}
\bigl(S_1\bigr)\bigl(1+\sum_{i=1}^{\infty}\lambda^{i}\int_{0}^{b}\bar{G}^{\ast i}(y)dy\bigr)},&x\in(0,1]\\
\dfrac{\lambda\bar{G}(x)+\lambda\sum_{i=1}^{\infty}\lambda^{i}\int_{0}^{b}\bar{G}(x-y)
\bar{G}^{\ast i}(y)  dy}{1+\lambda\mathbb{E}\bigl(S_1\bigr)\bigl(1+\sum_{i=1}^{\infty}\lambda^{i}\int_{0}^{b}
\bar{G}^{\ast i}(y)dy\bigr)}, &x\in (1,\infty)
\end{array}\right.
\end{equation}
\end{theorem}

\begin{proof} 
The condition $\mathbb{E}\bigl(S_1\bigr)<\infty$ ensures that the expected time between two consecutive 
visits of $V^{2,x}_t$ at level 1 is finite so that the limit theorem for regenerative processes can be applied. 
Setting the invariant density $\tilde{\pi}(x)$ equal to the upcrossing rate of level $x\in(0,1]$ we get 
\begin{eqnarray}
\tilde{\pi}(x)&=&\lambda\int_0^x \bar{G}(x-y)\pi(dy)=
\lambda\pi(0)\bar{G}(x)+\lambda\int_{0}^x\bar{G}(x-y)\tilde{\pi}(y)dy\nonumber\\
&=&\lambda\pi(0)\bar{G}(x)+\lambda\bar{G}\ast\tilde{\pi}(x)
\end{eqnarray}
As in the proof of Theorem 1 this yields 
for $x\in(0,1]$
\begin{equation}\label{eq1}
\tilde{\pi}(x)
=\lambda\pi(0)\bar{G}(x)+\lambda\bar{G}\ast\tilde{\pi}(x)=\pi(0)\sum_{i=1}^{\infty}\lambda^{i}\bar{G}^{\ast i}(x).
\end{equation}
For $x\in(1,\infty)$ the same arguments as above show that 
\begin{equation}\label{eq2}
\tilde{\pi}(x)=\lambda\pi(0)\bar{G}(x)+\lambda\int_{0}^{1}\bar{G}(x-y)\tilde{\pi}(y)dy.
\end{equation}
If we define $\tilde{\tilde{\pi}}(x)=\tilde{\pi}(x)\mathbf{1}_{(0,1]}(x)$, we obtain from \eqref{eq1} and
\eqref{eq2} that for all $x\in(0,\infty)$ we have
\begin{equation}\label{eq3}
\tilde{\pi}(x)=\lambda\pi(0)\bar{G}(x)+\lambda 
 \int_{0}^{x}\bar{G}(x-y)\tilde{\tilde{\pi}}(y)dy=\lambda\pi(0)\bar{G}(x)+\lambda (\bar{G}\ast\tilde{\tilde{\pi}})(x).
\end{equation}
Taking the integral in \eqref{eq3}, an application of Fubini's theorem and \eqref{eq1} leads to 
\begin{eqnarray}
1-\pi(0)&=&\lambda\pi(0)\mathbb{E}(S_1)+\lambda \mathbb{E}(S_1)\pi(1)=\lambda\pi(1)\mathbb{E}(S_1)\\
&=&\pi(0)\lambda\mathbb{E}\bigl(S_1\bigr)\bigl(1+\sum_{i=1}^{\infty}\lambda^{i}\int_{0}^{b}\bar{G}^{\ast i}(y)dy\bigr),
\end{eqnarray}
which yields
\begin{equation}\label{pi0}
\pi(0)=\frac{1}{1+\lambda\mathbb{E}\bigl(S_1\bigr)\bigl(1+\sum_{i=1}^{\infty}\lambda^{i}\int_{0}^{1}\bar{G}^{\ast i}(y)dy\bigr)}.
\end{equation}
The claim follows now from \eqref{pi0}, \eqref{eq1} and \eqref{eq2}.
\end{proof}



\end{subsection}

\begin{subsection}{A sufficient condition for geometric ergodicity}
For jump distributions with unbounded support $V^{2,x}_t$  is in general not geometrically ergodic. 
The next theorem gives a sufficient condition.  
\begin{theorem}
The process $V_t^{2,x}$ is geometrically ergodic if  
\begin{equation}
\mathbb{E}\bigl(r^{S_1}\bigr)< \infty \ \mbox{ for some } r>1.
\end{equation}
\end{theorem}
\begin{proof}
Let 
\begin{equation}
\tau_{C}^x=\inf_{t}\{t>0 :V_t^{2,\cdot}\in C\}.
\end{equation}
The proof is based on Theorem 15.0.1 in \cite{MeTw} which, translated to our setting, essentially states the following: 
If there exists a petite set $C\in\mathcal{B}(\mathbb{R}_+)$ 
(for a definition of the term `petite' we refer to \cite{MeTw}) and $r>1$ such
that 
\begin{equation}
\sup_{x\in C}\mathbb{E}_x\bigl(r^{\tau_C^x}\bigr)<\infty,
\end{equation}
then $V_t^{2,x}$ is geometrically ergodic.
Now we can choose $C=\{0\}$ and the claim follows.
\end{proof}

\end{subsection}

\begin{subsection}{Jump distributions with compact support}
In this subsection we assume that $G$ has compact support. Let $b$ be minimal such that
\begin{equation}
\mbox{supp}(G)\subset[0,b].
\end{equation}
By definition \eqref{defiV2} of the process $V^{2,x}_t$ it follows that
\begin{equation}
V^{2,x}_t\subset [0,b+1].
\end{equation}
 In order to estimate $d(t)$, let us bound $\bar{d}(t)$ for this example by using once 
again \eqref{coupling_inequalityI}, where $T^{x,y}$ is defined here is 
as before in the sense that in the former definition of $T^{x,y}$ one simply has to replace $V_t^{1,x}$ by $V_t^{2,x}$ and $V_t^{1,y}$ by $V_t^{2,y}$.\\ 

Let $0=x_0<x_1,\ldots<x_{N(\epsilon)-1}=b+1$ be a decomposition of the interval 
$[0,b+1]$ such that $x_{i+1}-x_{i}\le\epsilon$ for $i\in\{0,\ldots,N(\epsilon)-1\}$
\begin{eqnarray}
\bar{d}(t)&=&\sup_{x,y\in[0,b+1]}||\mathbb{P}\bigl(V_t^{1,x}\in \cdot\bigr)-\mathbb{P}\bigl(V_t^{1,y}\in \cdot\bigr)\|_{TV}\nonumber\\
&\le&\sum_{i=0}^{N(\epsilon)-1}\sup_{x,y\in[x_i,x_{i+1})}||\mathbb{P}\bigl(V_t^{1,x}\in \cdot\bigr)-\mathbb{P}\bigl(V_t^{1,y}\in \cdot\bigr)\|_{TV}\nonumber\\
&\le& \frac{b+1}{\epsilon}\bigl(1-e^{-\lambda(1+\epsilon)}\bigr)^{\lfloor\frac{t}{b+1}\rfloor}.
\end{eqnarray}
This implies
\begin{equation}
\limsup_{t\rightarrow\infty}-\frac{1}{t}\log\bar{d}(t)\ge -\frac{1}{b+1}\log\bigl(1-e^{-\lambda(1+\epsilon)}\bigr)\,\,\,\,\forall\epsilon>0,
\end{equation}
which immediately yields that
\begin{equation}\label{deterministic}
\limsup_{t\rightarrow\infty}-\frac{1}{t}\log d(t)\ge -\frac{1}{b+1}\log\bigl(1-e^{-\lambda}\bigr).
\end{equation} 
Therefore, 
\begin{equation*}
\lim_{t\rightarrow\infty} e^{\alpha t} d(t)=0 
\end{equation*} 
for every $\alpha   < \frac{1}{b+1}|\log\bigl(1-e^{-\lambda}\bigr)|$.

\end{subsection}

\end{section}




\begin{thebibliography}{9}
\bibitem{As}  
\textsc{Asmussen, S:} \textsc{Applied Probability and Queues}. 2nd ed., Springer (2003).

\bibitem{BPSZ} \textsc{O. Boxma, D. Perry, W. Stadje and S. Zacks:} 
The $M/G/1$ queue with quasi-restricted accessibility. 
 {\it Stoch. Models} 25, 151-196 (2009). 

\bibitem{brill} \textsc{Brill, P. H.:} \textit{Level Crossing Methods in Stochastic Models}. Springer (2008). 


\bibitem{Ca2}
\textsc{Callaert, H.:}
On the rate of convergence in birth-and-death processes.
\textit{Bull. Soc. Math. Belg.} 26, 173–184 (1974).

\bibitem{Ch}
\textsc{Chen, M. F.:}
Exponential $l^2$ convergence and $l^2$ spectral gap for Markov processes.
\textit{Acta Math. Sin.} 7, 19–37 (1991).



\bibitem{C1} \textsc{Cohen, J. W.:} Single server queue with restricted
accessibility. \textit{J.  Engineering Math.} {3}, 265-285 (1969).  



\bibitem{Co} 
\textsc{Cohen, J. W.:} 
{\it The Single Server Queue.} 2nd ed. North Holland (2003).

\bibitem{Da} 
\textsc{Daley, D. J.:} 
Single server queuing Systems with uniformly limited queuing times. 
\textit{J. Austral. Math. Soc.} 4, 347-358 (1964).

\bibitem{GaGo}
\textsc{Garmarnik, D., Goldberg, D.:}
On the rate of convergence to stationarity of the $M/M/N$ Queue in the 
Halfin-Whitt regime.
\textit{arXiv:1003.2004}.

	
\bibitem{GK} \textsc{Gnedenko, B. V. and Kovalenko, I. N.}:  \textit{Introduction
to
Queueing Theory}. 2nd ed., Birkh\"auser (1989).
 
\bibitem{GS} \textsc{B. Gavish, B. and Schweitzer, P. J.}: The Markovian queue with
bounded
waiting time. \textit{Management Science} {23}, 1349-1357 (1997).
 
\bibitem{Ho} \textsc{Hokstad, P.:} A single server queue with constant 
service time and restricted accessibility. \textit{Management Science} 
 {25}, 205-208 (1979).  
\bibitem{KaMc1}
\textsc{Karlin, S., McGregor, J.L.:} 
The  differential equations of birth-and-death processes, and the Stieltjes moment problem.
\textit{Trans. Amer. Math. Soc.} 85, 589-646 (1957).

\bibitem{KaMc2}
\textsc{Karlin, S., McGregor, J. L.:}
Many-server queueing processes with Poisson input and exponential service times. {\it Pacific J. Math.} 8, 87–118 (1958).



\bibitem{LePeWi}
\textsc{Levin, D. A., Peres, Y., Wilmer, E. L.:} 
{\sl Markov Chains and Mixing Times}. American Mathematical Society (2008).






\bibitem{LT} \textsc{Loris-Teghem, J.:} On the waiting time distribution in a 
generalized queueing system with uniformly bounded sojourn time. 
\textit{J. Appl. Prob.} {9}, 642-649 (1972).


\bibitem{LuMeTw}
\textsc{Lund, R.B., Meyn, S.P., Tweedie, R.L.:}
Computable exponential convergence rates for stochastically ordered Markov processes. \textit{Ann. Appl. Probab.} 6,
218-237 (1996).


\bibitem{LuTw}
\textsc{Lund, R. B., Tweedie, R. L.:}
Geometric convergence rates for stochastically ordered Markov chains.
\textit{Math. Oper. Res.}, 21, 182-194 (1996).



\bibitem{MeTw}
\textsc{Meyn, S.P., Tweedie, R.L.:}
{\sl Markov Chains and Stochastic Stability}. Springer (1993).


\bibitem{Ma}
\textsc{Mao, Y. H.:}
Strong ergodicity for Markov processes by coupling methods.
\textit{J. Appl. Probab.} 39, 839-852 (2002).


\bibitem{Nu} 
\textsc{Nummelin, E.:} {\sl General Irreducible Markov Chains and Non-Negative Operators}. Cambridge Univ. Press (1984).


\bibitem{PeStaZa} 
\textsc{Perry, D., Stadje, W., Zacks, S.:} A duality approach to queues with service restrictions and storage systems
with state-dependent rates. \textit{Preprint} (2010).

\bibitem{RoTw}
\textsc{Roberts, G. O., Tweedie, R. L.:}
Rates of convergence of stochastically monotone and continuous time Markov models. 
\textit{J. Appl. Probab.} 37, 359-373 (2000). 
 
\bibitem{ScTw}
\textsc{Scott, D. J., Tweedie, R. L.:}
Explicit rates of convergence of stochastically monotone Markov chains.
In
\textit{Proceedings of the Athens Conference on Applied Probability and Time Series: 
Papers in Honour of J. M. Gani and E. J. Hannan} (C. C. Heyde, Yu. V. Prohorov, R. Pyke and S.T. Rachev, eds.), 176-191 
(1996). 

\bibitem{StPa}
\textsc{Stadje, W., Parthasarathy, P.R.:}
On the convergence to stationarity of the many-server Poisson queue.
\textit{J. Appl. Probab.} 36 , 546-557 (1999).

\bibitem{Th1}
\textsc{Thorisson, H.:}
The coupling of regenerative processes.
\textit{Adv. Appl. Prob.} 15, 531-561 (1983).


\bibitem{Th2}
\textsc{Thorisson, H.:}
The queue $GI/G/1$: finite moments of the cycle variables and uniform rates of convergence.
\textit{Stoch. Proc. Appl.} 19, 85-99 (1985).



\bibitem{Va1}
\textsc{Van Doorn, E. A.:}
Conditions for exponential ergodicity and bounds for the decay parameter of a birth-death process.
\textit{Adv. Appl. Prob.} 17, 514-530 (1985).   


\bibitem{VaZe1}
\textsc{Van Doorn, E. A., Zeifman, A. I.:}
On the speed of convergence to stationarity of the Erlang loss system.
\textsc{Queueing Systems} 63, 241–252 (2009).

\bibitem{VaZePa}
\textsc{Van Doorn, E. A., Zeifman, A. I., Panfilova, T. L.:}
Bounds and asymptotics for the rate of convergence of birth-
death processes.
\textit{Teor. Veroyatnost. I Primenen.} 54, 18-38 (2009).




\end{thebibliography}
\end{document}